\documentclass[twoside,11pt]{amsart}

\usepackage{a4wide}
\usepackage{latexsym}
\usepackage{mathrsfs}
\usepackage[T1]{fontenc}
\usepackage{enumitem}
\usepackage{mathrsfs}

\usepackage{amssymb}
\usepackage{latexsym}
\usepackage{amsmath}
\usepackage{fancyhdr}
\usepackage{bbm}
\usepackage{setspace}
\usepackage{mathabx}

\usepackage{amsaddr}

\numberwithin{equation}{section}

\newtheorem{theorem}{Theorem}[section] 
\newtheorem{lemma}[theorem]{Lemma}     
\newtheorem{corollary}[theorem]{Corollary}
\newtheorem{proposition}[theorem]{Proposition}
\theoremstyle{definition}

\newtheorem{definition}[theorem]{Definition}

\newcommand {\C}{\mathbb C}

\newcommand {\N}{\mathbb N}

\newcommand {\K}{\mathcal K}
\newcommand {\B}{\mathcal B}
\newcommand {\A}{\mathcal A}

\newcommand{\supp}{{\rm supp\,}}

\newcommand{\eps}{\varepsilon}
\newcommand{\spn}{{\rm span \,}}

\newcommand{\im}{\operatorname{im}}
\newcommand{\id}{\operatorname{id}}
\newcommand{\dd}{{\rm \, d}}

\newcommand{\SUB}{\operatorname{\mathbf{SUB}}}
\newcommand{\LID}{\operatorname{\mathbf{CLI}}}
\newcommand{\RID}{\operatorname{\mathbf{CRI}}}

\title{Left Ideals of Banach Algebras and Dual Banach Algebras}

\begin{document}

\author[J.\ T.\  White]{Jared T.\ White}
\address{
Laboratoire de Math{\'e}matiques de Besan{\c c}on\\
Universit{\'e} de Franche-Comt{\'e}\\
16 Route de Gray\\
25030 Besan{\c c}on\\
France}
\email{jw65537@gmail.com}

\date{2018}

\subjclass[2010]{Primary: 16P40, 46H10; secondary: 43A10, 43A20, 47L10}
\keywords{Noetherian, Banach algebra, dual Banach algebra, Multiplier algebra, weak*-closed ideal}

\begin{abstract}
We investigate topologically left Noetherian Banach algebras. We show that if $G$ is a compact group, then $L^{\, 1}(G)$ is topologically left Noetherian if and only if $G$ is metrisable. We prove that, given a Banach space $E$ such that $E'$ has BAP, the algebra of compact operators $\K(E)$ is topologically left Noetherian if and only if $E'$ is separable; it is topologically right Noetherian if and only if $E$ is separable. We then give some examples of dual Banach algebras which are topologically left Noetherian in the weak*-topology. Finally we give a unified approach to classifying the weak*-closed left ideals of certain dual Banach algebras that are also multiplier algebras, with applications to $M(G)$ for $G$ a compact group, and $\B(E)$ for $E$ a reflexive Banach space with AP.
\end{abstract}

\maketitle

\section{Introduction}
\noindent
Those studying Banach algebras have long been interested in the interplay between abstract algebra and abstract analysis. This motivates the comparison of the following two definitions: 

\begin{definition} 		\label{def1}
Let $A$ be a Banach algebra, and let $I$ be a closed left ideal of $A$.
	\begin{enumerate}
		\item[{\rm (i)}] Given $n \in \N$, we say that $I$ is \textit{(algebraically)        			generated} by 
			$x_1, \ldots, x_n \in I$ if 
			$$I = A^\sharp x_1+ \cdots + A^\sharp x_n.$$
			When there exist such elements 
			$x_1, \ldots, x_n$ for some $n \in \N$ we say that $I$ 
			is \textit{(algebraically) finitely-generated}.
		\item[{\rm (ii)}] Given $n \in \N$, we  say that $I$ is 
			\textit{topologically generated} by 
			$x_1, \ldots, x_n \in I$ if 
			$$I = \overline{A^\sharp x_1+ \cdots + A^\sharp x_n}.$$
			When there exist such elements 
			$x_1, \ldots, x_n$ for some $n \in \N$ we say that $I$ 
			is \textit{topologically finitely-generated}.
	\end{enumerate}
\end{definition}

\noindent
Here $A^\sharp$ denotes the unitisation of $A$.

It seems both natural and obvious that Definition \ref{def1}(ii) is the appropriate one for the normed setting since it takes account of the topology. 
However, as Banach algebraists we wish to establish a more precise picture of exactly how these two definitions play out. One result in this spirit is the beautiful theorem of Sinclair and Tullo from 1974 \cite{ST}:

\begin{theorem} 	\label{1.0}
Let $A$ be a Banach algebra which is (algebraically) left Noetherian. Then $A$ is finite-dimensional.
\end{theorem}

In recent years a number of papers have appeared which further illustrate that algebraic finite-generation of left ideals in Banach algebras is a very strong condition. This has been motivated by a conjecture of Dales and {\.Z}elazko \cite{DZ} which states that a unital Banach algebra in which every maximal left ideal is finitely-generated is finite-dimensional. The conjecture is known to hold for many classes of Banach algebras \cite{DZ, DKKKL, BK, W1}. For example, in \cite{W1} the present author verified the conjecture for many of the algebras coming from abstract harmonic analysis, including the measure algebra of a locally compact group, as well as a large class of Beurling algebras.

In this article we aim to fill in another corner  of this picture and contrast with the above results by investigating topologically left Noetherian Banach algebras, which we define as follows:

\begin{definition}  \label{1.1}
Let $A$ be a Banach algebra. We say that $A$ is \textit{topologically left Noetherian} if every closed left ideal is topologically finitely-generated.
\end{definition}

We shall demonstrate below that, in contrast to the Sinclair--Tullo Theorem, there are many infinite-dimensional examples of topologically left Noetherian Banach algebras, and that, moreover, the condition often picks out a nice property of some underlying group or Banach space.

We note that versions of Noetherianity for topological rings and algebras are not at all new, and various different versions of this notion exist: see e.g. \cite{C, M}.  Moreover, topological left Noetherianity as in Definition \ref{1.1} was the subject of a post by Kevin Casto on Mathoverflow \cite{Casto}. His question is partially answered by our Theorem \ref{1.4}. We also mention that the fact that separable C*-algebras are topologically left Noetherian  has been known since Prosser's 1963 memoir \cite{Pr}, in which it was shown that every closed left ideal of a separable C*-algebra is topologically principal \cite[Corollary, pg. 26]{Pr}. It should be noted that using modern techniques a short proof of this fact can be obtained by using the correspondence between the closed left ideals of a C*-algebra and its hereditary C*-subalgebras, and the fact that, in the separable case, hereditary C*-subalgebras of a C*-algebra $A$ are all of the form $\overline{xAx}$, for some positive element $x \in A$.

One natural condition that we could consider for Banach algebras that is different from ours would be an ascending chain condition on chains of closed left ideals as in \cite[Proposition 4.1]{C}. However, we know of no infinite-dimensional examples of Banach algebras satisfying this condition. Indeed, many of the natural examples of Banach algebras satisfying Definition \ref{1.1} are easily seen to fail the ascending chain condition. These include $C[0,1]$, as well as the infinite-dimensional examples in Theorems \ref{1.2} and \ref{1.3} below. One might be tempted to try to prove that there are no infinite-dimensional Banach algebras satisfying an ascending chain condition for closed left ideals. However, if this were true, it would imply a negative solution to the question of whether there exists an infinite-dimensional, topologically simple, commutative Banach algebra, which is a notorious open question. For these reasons we have chosen only to study Definition \ref{1.1} in this article. 

We now outline the main results of the paper. In Section 2 we shall fix some notation and prove some general results that we shall require in the sequel. Section 3 is concerned with Banach algebras on groups, and our main result is the following:

\begin{theorem}		\label{1.2}
Let $G$ be a compact group. Then $L^{\, 1}(G)$ is topologically left Noetherian if and only if $G$ is metrisable.
\end{theorem}

We also prove an analogous result for the Fourier algebra of certain discrete groups (Proposition \ref{2.1.3a}).

In Section 4 we turn our attention to algebras of operators on a Banach space. We denote the set of compact operators on a Banach space $E$ by $\K(E)$. Our main result is the following:

\begin{theorem}		\label{1.3}
	\begin{enumerate}
		\item[\rm (i)] Let $E$ be a Banach space with AP. Then $\K(E)$ is topologically left Noetherian if and only if $E'$ is separable.
		\item[\rm (ii)] Let $E$ be a Banach space such that $E'$ has BAP. Then $\K(E)$ is topologically right Noetherian if and only if $E$ is separable.
	\end{enumerate}
\end{theorem}
\noindent
This theorem actually follows from more general results about the algebra of approximable operators $\A(E)$, for a formally larger class of Banach spaces $E$ (Theorem \ref{2.3.8} and Theorem \ref{2.3.12}).

In Section 5 we consider weak*-topologically left Noetherian dual Banach algebras, believing that this will be a more appropriate notion for the measure algebra $M(G)$ of a locally compact group $G$, and for $\B(E)$, the algebra of bounded linear operators on a reflexive Banach space $E$. We consider Banach algebras $A$ for which the multiplier algebra $M(A)$ is a dual Banach algebra in a natural way, and prove a general result, Proposition \ref{2.0.9}, which says that $M(A)$ is weak*-topologically left Noetherian whenever $A$ is topologically left Noetherian. We then apply this theorem to the algebras of the form $M(G)$ and $\B(E)$ to get the following corollary:

\begin{corollary}		\label{1.4}
	\begin{enumerate}
		\item[\rm (i)] Let $G$ be a compact, metrisable group. Then $M(G)$ is weak*-to\-p\-o\-logically left Noetherian.
		\item[\rm (ii)] Let $E$ be a separable, reflexive Banach space with AP. Then $\B(E)$ is weak*-topologically left and right Noetherian.
\end{enumerate}
\end{corollary}

We then consider a more restricted class of Banach algebras, and we formulate an abstract approach for relating the ideal structure of a Banach algebra $A$ belonging to this class to the weak*-ideal structure of $M(A)$ (Theorem \ref{2.0.11b}). In Proposition \ref{2.0.11} we use this to show that, for this class, weak*-topological left Noetherianity of $M(A)$ is equivalent to a $\Vert \cdot \Vert$-topological condition on $A$. In Section 6 we demonstrate how Theorem \ref{2.0.11b} gives a unified strategy for classifying the weak*-closed left ideals of both $M(G)$, for $G$ a compact group, and $\B(E)$, for $E$ a reflexive Banach space with AP. We then observe that this leads to classifications of the closed right submodules of the predules.

Finally, we mention that in \cite{LW} the author of the present work together with Niels Laustsen will show that there is a certain reflexive Banach space $E$ which has AP such that $\B(E)$ fails to be weak*-topologically left Noetherian. This is the only example of a dual Banach algebra that we know of that fails to have this property.

\section{Preliminaries}		\label{S3.2}
\noindent
We first fix some general notation. Given a locally compact group $G$, we denote by $L^{\, 1}(G)$ the Banach algebra of integrable functions on $G$, which we refer to as the group algebra of $G$. We write $M(G)$ for the  Banach algebra of complex, regular Borel measures on $G$, known here as the measure algebra. We write $A(G)$ for the Fourier algebra of $G$, as defined by Eymard in \cite{E}. We write $C(G)$ for the linear space of complex-valued functions on $G$, and $C_0(G)$ for the subspace consisting of functions vanishing at infinity. Of course, $C_0(G)$ is a Banach space with the supremum norm; when $G$ is compact $C(G) = C_0(G)$, and we prefer the former notation over the latter.

By a representation of $G$ we implicitly mean a continuous unitary representation of $G$ on a Hilbert space. We write $\widehat{G}$ for the unitary dual of $G$, and a typical element of $\widehat{G}$ will be represented as $(\pi, H_\pi)$, where $H_\pi$ is a Hilbert space, and $\pi$ is an irreducible representation of $G$ on $H_\pi$. 
Given an arbitrary representation $(\pi, H_\pi)$ of $G$, and vectors $\xi, \eta \in H_\pi,$ we write $\xi*_\pi \eta$ for the function on $G$ defined by $t \mapsto \langle \pi(t) \xi, \eta \rangle \ (t \in G)$. 
We shall denote the modular function on $G$ by $\Delta$. We also use the notation  $\widecheck{f}(t) = f(t^{-1}) \ (t \in G)$ for $f \in L^{\, 1}(G)$.

For a Banach  algebra $A$ we denote by $A^\sharp$ the  (conditional) unitisation of $A$, and by $M(A)$ the multiplier algebra of $A$. We write $\LID(A)$ for the lattice of closed left ideals of $A$, and $\RID(A)$ for the lattice of closed right ideals. 
We denote the left action of $A$ on its dual by $a \cdot \lambda$ for $a \in A$ and $\lambda \in A'$, and set $A \cdot A' = \{a \cdot \lambda: a \in A, \ \lambda \in A' \}$.
 Usually $A$ will have a bounded approximate identity, in which case, by Cohen's
  factorisation theorem, $A \cdot A' = \overline{\spn}(A \cdot A')$. We use similar notation for the right action.

Given a Banach space $E$, we shall denote its dual space by $E'$. We write $\B(E)$ for the algebra of bounded linear operators $E \rightarrow E$. We write $\K(E)$ for the ideal of compact operators, $\A(E)$ for the approximable operators, and $\mathcal{F}(E)$ for the finite rank operators. We write $\SUB(E)$ for the lattice of closed linear subspaces of $E$. Given $x \in E$ and $\lambda \in E'$, we write $x \otimes \lambda$ for the rank one operator $y \mapsto \lambda(y)x \ (y \in E)$.

Given subsets $X \subset E$ and $Y \subset E'$, we use the notation
$$X^\perp = \{ \lambda \in E' : \langle x, \lambda \rangle = 0 \ (x \in X) \}, \qquad
Y_\perp = \{ u \in E : \langle u, \varphi \rangle = 0 \ (\varphi \in Y) \},$$
and we recall the well-known formulae
\begin{equation}		\label{eq2.1}
(X^\perp)_\perp = \overline{\spn}(X), \qquad (Y_\perp)^\perp = \overline{\spn}^{w^*}(Y).
\end{equation}

A Banach space $E$ is said to have the \textit{approximation property}, or simply \textit{AP}, if, whenever $F$ is another Banach space, we have 
$\A(F, E) = \K(F, E)$. 
There is also an equivalent formulation of the approximation property which has some useful generalizations: a Banach space $E$ has AP if and only if, for every compact subset $K \subset E$ and every $\eps >0$, there exists $T \in \mathcal{F}(E)$ such that $\Vert T x - x \Vert < \eps \ (x \in K)$ \cite[Theorem 3.4.32]{Meg}. We say that $E$ has the \textit{bounded approximation property}, or \textit{BAP}, if there exists a constant $C>0$ such that the operator $T$ can be chosen to have norm at most $C$. Clearly BAP implies AP.
Moreover, a reflexive Banach space with AP has BAP \cite[Theorem 3.7]{Cas}.
Many Banach spaces have the bounded approximation property: for instance any Banach space with a Schauder basis \cite[Theorem 4.1.33]{Meg} has BAP, and it can be deduced from this that any Hilbert space has BAP. The Banach space $\B(H)$, for $H$ an infinite-dimensional Hilbert space, does not even have AP \cite{Sz}.

We recall from \cite{R} that a \textit{dual Banach algebra} is a Banach algebra $A$ which is isomorphically a dual Banach space, in such a way that the multiplication is separately weak*-continuous. Equivalently, a Banach algebra with (isomorphic) predual $X$ is a dual Banach algebra if and only if $X$ may be identified with a closed $A$-submodule of $A'$. Examples include the measure algebra $M(G)$ of a locally compact group $G$, with predual $C_0(G)$, as well as $\B(E)$ for a reflexive Banach space $E$, with predual given by $E \widehat{\otimes} E'$, where 
$\widehat{\otimes}$ denotes the projective tensor product of Banach spaces. 

Recall that a \textit{semi-topological algebra} is a pair $(A, \tau)$, where $A$ is an algebra, and $\tau$ is a topology on $A$ such that $(A, +, \tau)$ is a topological vector space, and such that multiplication on $A$ is separately continuous. 
For example, a dual Banach algebra with its weak*-topology is a semi-topological algebra.

Let $(A, \tau)$ be a semi-topological algebra. Let $I$ be a closed left ideal of$A$, and let $n \in \N$. We say that $I$ is \textit{$\tau$-topologically generated} by elements $x_1, \ldots, x_n \in I$ if 
$$I = \overline{A^\sharp x_1+ \cdots +A^\sharp x_n}.$$
We say that $I$ is \textit{$\tau$-topologically finitely-generated} if there exist $n \in \N$ and $x_1, \ldots, x_n \in I$ which $\tau$-topologically generate $I$. We say that $I$ is \textit{$\tau$-topologically principal} if $I$ is of the form $\overline{A^\sharp x}$, for some $x \in I$. We say that $A$ is \textit{$\tau$-topologically left Noetherian} if every closed left ideal of $A$ is $\tau$-topologically finitely-generated. For example, we shall often discuss weak*-topologically left Noetherian dual Banach algebras. When the topology is the norm topology on a Banach algebra we may simply speak of ``topologically finitely-generated left ideals'' et cetera. 

Analogously we may define \textit{$\tau$-topologically finitely-generated right ideals}, as well as \textit{$\tau$-topologically right Noetherian algebras}.  If the algebra in question is commutative we usually drop the words ``left'' and ``right''.

We note that when a semi-topological algebra $A$ has a left approximate identity we have 
$$\overline{A^\sharp x_1 + \cdots + A^\sharp x_n} = \overline{Ax_1 + \cdots + Ax_n},$$
  for each $n \in \N,$ and each $x_1, \ldots, x_n \in A$. When this is the case we usually drop the unitisations in order to ease notation. For example, in the proof of Theorem \ref{2.1.2} below, 
  we shall write
   $\overline{L^{\, 1}(G)*g}$ in place of $\overline{L^{\, 1}(G)^\sharp*g}$, for $G$ a locally compact group and $g \in L^{\, 1}(G)$.

The following lemma will be invaluable throughout this article.

\begin{lemma} \label{2.1.1}
Let $A$ be a semi-topological algebra with a left approximate identity. Let $J$ be a dense right ideal of $A$. Then $J$ intersects every closed left ideal of $A$ densely.
\end{lemma}

\begin{proof}
Let $(e_\alpha)$ be a left approximate identity for $A$, which we may assume belongs to $J$ (if not, then for each open neighbourhood of the origin $U,$ and each index $\alpha$, choose $f_{\alpha, U} \in J$ such that $e_\alpha - f_{\alpha, U} \in U$. Then $(f_{\alpha, U})$ is easily seen to be a left approximate identity for $A$.) Let $I$ be a closed left ideal of$A$ and let $a \in I$. Then for every index $\alpha$ we have $e_\alpha a \in J \cap I$. Since $a = \lim_{\alpha} e_{\alpha} a \in \overline{J \cap I}$, and $a$ was arbitrary, it follows that $\overline{J \cap I} = I$, as required.
\end{proof}

Next we show that $\tau$-topological left Noetherianity is stable under taking quotients and extensions.
For this lemma only we drop the $\tau$s and write ``topologically left Noetherian'' et cetera, even though we are not necessarily talking about a topology induced by a norm.

\begin{lemma}		\label{2.1.1a}
Let $A$ be a semi-topological algebra, and let $I$ be a closed (two-sided) ideal of$A$.
	\begin{enumerate}
		\item[\rm (i)] If $A$ is topologically left Noetherian then so is $A/I$.
		\item[\rm (ii)] Suppose that both $I$ and $A/I$ are topologically left Noetherian. Then so is $A$.
		\item[\rm (iii)] A is topologically left Noetherian if and only if $A^\sharp$ is topologically left Noetherian.
	\end{enumerate}
\end{lemma}

\begin{proof}
Parts (i) and (ii) follow from routine arguments. For part (iii) we may suppose that $A$ is non-unital for otherwise the result is trivial. If $A$ is topologically left Noetherian then, since $A^\sharp / A \cong \C$ is topologically left Noetherian, it follows from (ii) that $A^\sharp$ is also. The converse follows from the fact that every closed left ideal of$A$ is also a closed left ideal of$A^\sharp$.
\end{proof}

\section{Examples From Abstract Harmonic Analysis}		\label{S3.3}
\noindent
In this section we shall prove Theorem \ref{1.2}. It is surely easiest to determine whether or not a Banach algebra is topologically left Noetherian when we know what its closed left ideals are. Fortunately, this is the case for the group algebra of a compact group, as well as for the Fourier algebra of certain discrete groups, including all amenable groups. As a sort of warm up for the proof of Theorem \ref{1.2} we shall show that, for such groups, the Fourier algebra $A(G)$ is topologically Noetherian if and only if $G$ is countable. Both proofs involve similar ideas.

\begin{proposition}		\label{2.1.3a}
Let $G$ be a discrete group such that $f \in \overline{A(G)f}$ for all $f \in A(G)$. Then $A(G)$ is topologically Noetherian if and only if $G$ is countable.
\end{proposition}

\begin{proof}
Given $E \subset G$, write $I(E) = \{f \in A(G) : f(x) = 0, \ x \in E \}$. By \cite[Proposition 2.2]{KL} the closed ideals of $A(G)$ are all of the from $I(E)$ for some subset $E$ of $G$.

Suppose first that $G$ is countable and let $I \triangleleft A(G)$ be closed. Let $E \subset G$ be such that $I = I(E)$, and enumerate $G \setminus E = \{ x_1, x_2, \ldots, \}$. Define
$g = \sum_{n=1}^\infty \frac{1}{n^2} \delta_{x_n} \in A(G)$. It is clear that $\supp g = G \setminus E$, and hence that
$$\left\{ x \in G : f(x) = 0 \text{ for every } f \in \overline{A(G)^\sharp g} \right\} = E.$$
It follows from the classification of the  closed ideals of $A(G)$ given above that 
$I = \overline {A(G)^\sharp g}$. As $I$ was arbitrary we conclude that $A(G)$ is topologically Noetherian.

Now suppose that $A(G)$ is topologically Noetherian. 
 Then there exist $n \in \N$ and $h_1, \ldots, h_n \in A(G)$ such that $A(G) = \overline{A(G)^\sharp  h_1 + \cdots + A(G)^\sharp  h_n}$. Since $A(G) \subset c_0(G)$, every function in $A(G)$ must have countable support. Hence $\mathcal{S} := \bigcup_{i=1}^n \supp h_i$ is a countable set. Every $f \in A(G)^\sharp  h_1 + \cdots + A(G)^\sharp  h_n$ has $\supp f \subset \mathcal{S}$, and of course, after taking closures, we see that this must hold for every $ f \in A(G)$. This clearly forces $\mathcal{S} = G$, so that $G$ must be countable.
\end{proof}

\textit{Remark.} The hypothesis of the previous proposition is satisfied by any discrete, amenable group, since in this case $A(G)$ has a bounded approximate identity by Leptin's Theorem, as well as many other groups including the free group on $n$ generators for each $n \in \N$ \cite{S}. The question of whether there are any locally compact groups which do not satisfy $f \in \overline{fA(G)}$ for every $f \in A(G)$ is considered a difficult open problem.
\vskip 2mm

We now recall some facts about compact groups. Firstly, for $G$ a compact group
the closed left ideals of $L^{\, 1}(G)$  have the following characterisation \cite[Theorem 38.13]{HR2}:

\begin{theorem} 	\label{2.1.2a}
Let $G$ be a compact group, and let $I$ be a closed left ideal of $L^{\, 1}(G)$. Then there exist linear subspaces $E_\pi \subset H_\pi \ (\pi \in \widehat{G})$ such that 
$$I = \left\{ f \in L^{\, 1}(G) : \pi(f)(E_\pi) = 0, \ \pi \in \widehat{G} \right\}.$$ 
\end{theorem}

Let $G$ be a compact group. Given $\pi \in \widehat{G}$ we write $T_{\pi}(G) = \spn \{ \xi *_\pi \eta : \xi, \eta \in H_\pi \}$, and we write 
$T(G) = \spn \{ \xi *_\pi \eta : \xi, \eta \in H_\pi, \pi \in \widehat{G} \}.$ We recall the following facts about these spaces from \cite{HR1, HR2}:

\begin{theorem}		\label{2.1.2b}
Let $G$ be a compact group. 
\begin{enumerate}
	\item[\rm (i)] Let $\sigma, \pi \in \widehat{G}$ with $\sigma \neq \pi$. Then $\sigma(\xi *_\pi \eta) = 0$.
	\item[\rm (ii)] The linear space $T(G)$ is a dense ideal in $L^{\, 1}(G)$.
	\item[\rm (iii)] For each $\pi \in \widehat{G}$ the space $T_\pi(G)$ is an ideal in $L^{\, 1}(G)$, and as an algebra $T_\pi(G) \cong M_{d_\pi}(G)$,
	where $d_\pi$ denotes the dimension of $H_\pi$.
\end{enumerate}
\end{theorem}

\begin{proof}
The formula 
$$\langle \sigma(f) \zeta_1, \zeta_2 \rangle = \int_G f(t) \langle \sigma(t) \zeta_1, \zeta_2 \rangle {\rm \, d}t,$$
for $f \in L^{\, 1}(G), \sigma \in \widehat{G}$ and $\zeta_1, \zeta_2 \in H_\sigma$ is well-known. Part (i) follows from this and the orthogonality 
relations \cite[Theorem 27.20 (iii)]{HR2}. Part (ii) follows from \cite[Theorem 27.20, Lemma 31.4]{HR2}, and part (iii) follows from \cite[Theorem 27.21]{HR2}.
\end{proof}

We now prove the main theorem of this section, Theorem \ref{2.1.2}. Observe that Theorem \ref{1.2} is simply ``(a) if and only if (c)''. The equivalence of conditions (b) and (c) has surely been noticed before,
but we include the proof to make our argument more transparent.

\begin{theorem} 	\label{2.1.2}
Let $G$ be a compact group. Then the following are equivalent:
\begin{enumerate}
\item[{\rm (a)}] $L^{\, 1}(G)$ is topologically left Noetherian;
\item[{\rm(b)}] $\widehat{G}$ is countable;
\item[{\rm(c)}] $G$ is metrisable.
\end{enumerate}
\end{theorem}

\begin{proof}
We first demonstrate that (b) implies (c). Our method is to show that $G$ is first-countable, which will imply that $G$ is metrisable by \cite[Theorem 8.3]{HR1}. 
Indeed, it follows from Tannaka--Krein duality \cite{JS} that the topology on $G$ is the initial topology induced by its irreducible continuous unitary representations, and as such has a base given by sets of the  form
$$U(\pi_1, \ldots, \pi_n;\eps;t):= \{s \in G: \Vert \pi_i(t) - \pi_i(s)\Vert < \eps, i = 1, \ldots, n \},$$
where $\eps>0, t \in G,$ and $ (\pi_1, H_1), \ldots, (\pi_n, H_n) \in  \widehat{G}$.
Hence, if $\widehat{G}$ is countable, for every $t \in G$ the sets $U(\pi_1, \ldots, \pi_n;1/m;t) \ (m \in \N, \pi_1, \ldots, \pi_n \in \widehat{G})$ form a countable neighbourhood base at $t$,
and so $G$ is first-countable.

Now suppose instead that $G$ is metrisable. Then $C(G)$ is separable. Since the infinity norm dominates the $L^{\, 2}$-norm for a compact space, and since $C(G)$ is dense in $L^{\, 2}(G)$, it follows that $L^{\, 2}(G)$ is separable. By \cite[Theorem 27.40]{HR2} 
$$L^{\, 2}(G) \cong \bigoplus_{\pi \in \widehat{G}} H_\pi^{\oplus {\rm dim} H_\pi}, $$
which is clearly separable only if $\widehat{G}$ is countable. Hence (c) implies (b).

Next we show that (b) implies (a). Suppose that $\widehat{G}$ is countable. By Theorem \ref{2.1.2b}(ii) 
$T(G)$ is a dense ideal in $L^{\, 1}(G)$ so that, by Lemma \ref{2.1.1}, $\overline{I \cap T(G)} = I$ for every closed left ideal $I$ in $L^{\, 1}(G)$. 

Fix a closed left ideal $I$ in $L^{\, 1}(G)$. By Theorem \ref{2.1.2a} there exist linear subspaces 
$E_\pi \subset H_\pi \ (\pi \in \widehat{G})$ such that 
$$I = \left\{ f \in L^{\, 1}(G) : \pi(f)(E_\pi) = 0, \ \pi \in \widehat{G} \right\}.$$
By Theorem \ref{2.1.2b}(iii), for each $\pi \in \widehat{G}$ we have $T_\pi(G) \cong M_{d_\pi}(\C)$, 
where $d_\pi$ is the dimension of $H_\pi$, and since $I \cap T_\pi(G)$ is a left ideal in $T_\pi(G)$ 
there must be an idempotent $P_\pi \in T_\pi(G)$ such that $I \cap T_\pi(G) = T_\pi(G) * P_\pi$. 
Set $\alpha_\pi = \Vert P_\pi \Vert^{-1}$ if $P_\pi \neq 0$, and set $\alpha_\pi = 0$ otherwise. Enumerate $\widehat{G} = \{ \pi_1, \pi_2, \ldots \}$, and define
$$g = \sum_{i=1}^\infty \frac{1}{i^2}\alpha_{\pi_i} P_{\pi_i} \in L^{\, 1}(G),$$
which belongs to $I$ because each $P_{\pi_i}$ does, and $I$ is closed.

We claim that $I = \overline{L^{\, 1}(G)*g}$. Indeed, $I \supset \overline{L^{\, 1}(G)*g}$ because $g \in I$. For the reverse inclusion we show that, for $j \in \N$ and $\xi \in H_{\pi_j}$, we have $\pi_j(f)(\xi) = 0$ for all $f \in \overline{L^{\, 1}(G)*g}$  
if and only if $\xi \in E_{\pi_j}$. The claim then follows from Theorem \ref{2.1.2a}. 
Indeed, if $f \in \overline{L^{\, 1}(G)*g}$ then $\pi_j(f)(\xi) = 0$ because $f \in I$. On the other hand if $\xi \in H_{\pi_j} \setminus E_{\pi_j}$
 then $\pi_j(P_{\pi_j})(\xi) \neq 0$, whereas $\pi_i(P_{\pi_j}) = 0$ for $i \neq j$ by Theorem \ref{2.1.2b}(i), which implies that $\pi_j(g)\xi = \frac{1}{j^2} \alpha_{\pi_j} \pi_j(P_{\pi_j})(\xi) \neq 0$.
This establishes the claim.

Finally we show that (a) implies (b). Assume that $L^{\, 1}(G)$ is topologically left Noetherian. Then there exist $r \in \N$ and $g_1, \ldots, g_r \in L^{\, 1}(G)$ such that 
$$L^{\, 1}(G) = \overline{L^{\, 1}(G)*g_1+ \cdots +L^{\, 1}(G)*g_r}.$$
For each $n \in \N$ there exist $t_n^{(i)} \in T(G) \ (i=1, \ldots, r)$ such that 
$$\Vert t_n^{(i)} - g_i \Vert < \frac{1}{n} \quad (i = 1, \ldots, r).$$
Let $\mathcal{S}$ be the set
$$\mathcal{S} = \left\{ \pi \in \widehat{G} : \text{ there exist } i, n \in \N \text{ such that } \pi \left( t_n^{(i)} \right) \neq 0 \right\}.$$
 We see that $\mathcal{S}$ is countable because, by Theorem \ref{2.1.2b}(i), each function 
$t_n^{(i)}$ satisfies $\pi(t_n^{(i)}) \neq 0$ for at most finitely many $\pi \in \widehat{G}$.
We shall show that $\mathcal{S} = \widehat{G}$. 

Assume instead that there exists some $\pi \in \widehat{G} \setminus \mathcal{S}$, and let $u$ be the identity element of $T_\pi(G)$. For $\sigma \in \widehat{G} \setminus \{ \pi \}$ we have $\sigma(u) = 0$, whereas $\pi \left(t_n^{(i)} \right) = 0$ for every $n \in \N$ and every $i =1, \ldots, r$. Hence $\sigma \left(t_n^{(i)}*u \right) = 0 \ (\sigma \in \widehat{G}, n \in \N, i \in \{ 1, \ldots, r \})$, which implies that $t_n^{(i)}*u = 0$ for every $n \in \N$ and $i = 1, \ldots, r$.

By taking the limit as $n$ goes to infinity, this shows that $g_i*u = 0 \ (i=1, \ldots, r)$, and hence that $f*u = 0$ for every $f \in L^{\, 1}(G).$ However, since $u$ was chosen to be an identity $u*u = u \neq 0$. This contradiction implies that $\widehat{G} = \mathcal{S}$, as claimed.
\end{proof}

The next proposition suggests to us that weak*-topological Noetherianity is a more interesting notion for the measure algebra of a locally compact group $G$ than $\Vert \cdot \Vert$-topological Noetherianity, and we explore this in the next section. Given a discrete group $G$ we write $\ell^{\, 1}(G)= L^{\, 1}(G)$ , and write $\ell^{\, 1}_{\, 0}(G)$ for its augmentation ideal, i.e. the maximal ideal consisting of those $f \in \ell^{\, 1}(G)$ such that $\sum_{t \in G} f(t) = 0$.

\begin{proposition} \label{2.1.3}
Let $G$ be a locally compact group such that $M(G)$ is topologically left Noetherian. Then $G$ is countable. If, in addition, $G$ is either compact or abelian, then $G$ is finite.
\end{proposition}

\begin{proof}
Suppose that $M(G)$ is topologically left Noetherian. Then, by Lemma \ref{2.1.1a} (i), so are its quotients, whence $\ell^{\, 1}(G_d)$ is topologically left Noetherian, where $G_d$ denotes the group $G$ with the discrete topology. It follows that 
$$\ell^{\, 1}_{\, 0}(G_d) = \overline{\ell^{\, 1}(G_d)*g_1 + \cdots + \ell^{\, 1}(G_d)*g_n},$$
 for some $n \in \N$ and some $g_1, \ldots, g_n \in \ell^{\, 1}_{\, 0}(G_d)$.
Let $H$ be the subgroup of $G$ generated by the supports of the functions $g_1, \ldots, g_n$. This is a countable set. 
Define $\sigma \colon \ell^{\, 1}(G_d) \rightarrow \C$ by 
$$\sigma \colon f \mapsto \sum_{x \in H} f(x) \quad (f \in \ell^{\, 1}(G_d)).$$
 Then, by the calculation performed in \cite[Lemma 3.6]{W1}, 
$$\sigma(f) = 0 \quad (f \in \ell^{\, 1}(G_d)*g_1 + \cdots + \ell^{\, 1}(G_d)*g_n),$$
 and hence, since $\sigma$ is clearly bounded, $\sigma(f) = 0$ for every $f \in \ell^{\, 1}_{\, 0}(G_d)$. This forces $G = H$. Hence $G$ is countable.

A countable locally compact group is always discrete, so that if it is also compact it must be finite. If $G$ is abelian, then the fact that $\ell^{\, 1}(G_d)$ is topologically Noetherian implies that $G$ is finite by \cite[Theorem 1.1]{A}.
\end{proof}

\section[Ideals of Approximable Operators]{Left and Right Ideals of Approximable Operators on a Banach Space}	\label{S3.5}
\noindent
In this section we shall prove Theorem \ref{1.3}. This will follow as a corollary of the formally more general Theorems \ref{2.3.8} and \ref{2.3.12} below.  Along the way we shall give a characterisation of
the closed right ideals of $\A(E)$, for $E$ any Banach space such that $\A(E)$ has as right approximate identity. This is analogous to the characterisation given by Gr{\o}nb{\ae}k in
\cite[Proposition 7.3]{G} of the closed left ideals of $\K(E)$, for a Banach space $E$ with the approximation property. We shall observe below that Gr{\o}nb{\ae}k's proof actually goes through,
with $\A(E)$ in place of $\K(E)$, under the formally  weaker hypothesis that $\A(E)$ has a left approximate identity.

Let $E$ be a Banach space, and $X \subset \B(E)$. Then we write
$$E' \circ X := \{ \lambda \circ T : T \in X \} = \bigcup_{T \in X} \im T'.$$
 Let $A$ be a closed subalgebra of $\B(E)$. Given closed linear subspaces $F \subset E'$ and $D \subset E$ we define
\begin{equation}	\label{eq2.4.0a}
\mathscr{L}_A(F) = \{T \in A: \im T' \subset F \}
\end{equation}
and
\begin{equation}	\label{eq2.4.0b}
\mathscr{R}_A(D) = \{ T \in A : \im T \subset D \}.
\end{equation}
These define families of closed left and right ideals respectively. We also define a family of closed left ideals by
\begin{equation}	\label{eq2.4.0c}
\mathscr{I}_A(D) = \{ T \in A : \ker T \supset D \},
\end{equation}
where $D$ is a closed linear subspace of $E$.
 When the ambient algebra $A$ is unambiguous we shall often drop the subscript and simply write $\mathscr{L}(F)$, $\mathscr{R}(D)$, and $\mathscr{I}(D)$. Usually $A$ will be either $\A(E)$ or $\B(E)$.
We shall show that when $\A(E)$ has a right approximate identity every closed right ideal of $\A(E)$ has the form $\mathscr{R}(D)$, for some closed linear subspace $D$ of $E$ (Theorem \ref{2.3.9}). We can restate Gr{\o}nb{\ae}k's result in a similar fashion:

\begin{theorem} 	\label{2.3.4}
Let $E$ be a Banach space such that $\A(E)$ has a left approximate identity. Then the map
$$ (\SUB(E'), \subset) \rightarrow (\LID(\A(E)), \subset),
 \quad F \mapsto \mathscr{L}(F)$$
 is a lattice isomorphism, with inverse given by 
$$ I \mapsto E'\circ I, \quad (I \in \LID(\A(E))).$$
\end{theorem}

\begin{proof}
Observe that, by Lemma \ref{2.1.1}, every closed left ideal of $\A(E)$ intersects densely
 with the finite rank operators. It follows from this that Gr{\o}nb{\ae}k's proof
  \cite[Proposition 7.3]{G} goes through under our hypothesis on $E$. We claim that, 
  in Gr{\o}nb{\ae}k's notation, $\Phi(F) = \mathscr{L}(F)$ and $\Psi(I) = E' \circ I$:
   showing each inclusion is routine, except $\Phi(F) \supset \mathscr{L}(F)$. For this,
    we again use the fact that $\mathcal{F}(E) \cap \mathscr{L}(F)$ is dense in $\mathscr{L}(F)$ to see that it is sufficient to check that, for a finite rank operator 
 $T = \sum_{i=1}^n x_i \otimes \lambda_i \in \mathscr{L}(F)$, we have $T' \in
 \overline{\spn}\{x \otimes \lambda : x \in E, \lambda \in F \}.$ Indeed, we may
  assume that $x_1, \ldots, x_n$ are linearly independent, and choose 
  $ \eta_j \in E' \ (j = 1, \ldots, n)$  such that 
  $\langle \eta_j, x_i \rangle = \delta_{ij} \ (i, j = 1, \ldots, n)$. 
  It then follows that $T'(\eta_i) = \lambda_i \in F \ (i= 1, \ldots, n)$, so that $T$ has the required form.
\end{proof}

We begin by addressing the topological left Noetherianity question for $\A(E)$.

\begin{lemma}	\label{2.3.5}
Let $E$ be a Banach space. Let $n \in \N$, let $T_1, \ldots, T_n \in \A(E)$, and let $I = \overline{\A(E)^\sharp T_1 + \cdots + \A(E)^\sharp T_n}$. Then 
$$E' \circ I = \overline{\im T_1' + \cdots + \im T_n'}.$$
\end{lemma}

\begin{proof}
As $E' \circ I = \bigcup_{T \in I} \im T'$ we have $E' \circ I \supset \im T_i' \ (i=1, \ldots n)$. Since $E' \circ I$ is a closed linear subspace, it follows 
that $E' \circ I \supset \overline{\im T_1' + \cdots + \im T_n'}$.

For the reverse inclusion, let $S \in I$ and let $\lambda \in E'$. There are sequences 
$$(R_1^{(j)})_j, \ldots, (R_n^{(j)})_j \subset \mathcal{A}(E) + \C \id_E$$
 such that
$$S = \lim_{j \rightarrow \infty} \left( R_1^{(j)} \circ T_1+ \cdots +R_n^{(j)} \circ T_n \right).$$
Then 
\begin{align*}
\lambda \circ S &= \lim_{j \rightarrow \infty} \left( \lambda \circ (R_1^{(j)} \circ T_1) + \cdots + \lambda \circ (R_n^{(j)} \circ T_n ) \right) \\
&= \lim_{j \rightarrow \infty} \left( T_1'(\lambda \circ R_1^{(j)}) + \cdots + T_n'(\lambda \circ R_n^{(j)}) \right) \in \overline{\im T_1' + \cdots + \im T_n'}.
\end{align*}
As $\lambda$ and $S$ were arbitrary, this concludes the proof.
\end{proof}

The next lemma gives a partial characterisation of when $\A(E)$ is topologically left Noetherian. The full characterisation will be given in Theorem \ref{2.3.8}.

\begin{lemma}	\label{2.3.6}
Let $E$ be a Banach space such that $\A(E)$ has a left approximate identity.
	\begin{enumerate}
		\item[{\rm (i)}] Let $F \subset E$ be a closed linear subspace. Then $\mathscr{L}(F)$ is topologically generated by $T_1, \ldots, T_n \in \A(E)$ if and only if
			\begin{equation}
			F = \overline{\im T_1' + \cdots + \im T_n'}.	\label{eq2.4.1}
			\end{equation}
		\item[{\rm (ii)}] The algebra $\A(E)$ is topologically left Noetherian if and only if every closed linear subspace of $E'$ has the form \eqref{eq2.4.1}, for some $n \in \N$ and $T_1, \ldots, T_n \in \A(E)$.
	\end{enumerate}
\end{lemma}

\begin{proof}
(i)	Suppose that $\mathscr{L}(F) = \overline{\A(E)T_1 + \cdots + \A(E)T_n}$, for some $T_1, \ldots, T_n \in \A(E)$. Then by Lemma \ref{2.3.5} 
$$E' \circ \mathscr{L}(F) = \overline{\im T_1' + \cdots + \im T_n'},$$
so that, by Theorem \ref{2.3.4}, $F = \overline{\im T_1' + \cdots + \im T_n'}.$

Conversely, suppose that there are maps $T_1, \ldots, T_n \in \A(E)$ such that  $F$ has the form \eqref{eq2.4.1}. Consider the left ideal 
$$I = \overline{\A(E)T_1 + \cdots + \A(E)T_n}.$$
 By Lemma \ref{2.3.5} we have $E' \circ I = F,$ and so by Theorem \ref{2.3.4} we have $I = \mathscr{L}(E' \circ I) = \mathscr{L}(F)$. Hence 
$$\mathscr{L}(F) = \overline{\A(E)T_1 + \cdots + \A(E)T_n},$$
as required.

(ii) This is clear from (i) and Theorem \ref{2.3.4}.
\end{proof}

In the proof of the next lemma we use the fact that every infinite-dimensional Banach space contains a basic sequence \cite[Theorem 4.1.30]{Meg}.

\begin{lemma}	\label{2.3.7}
Let $E$ be a Banach space, and let $F \subset E'$ be a closed, separable linear subspace. Then there exists $T \in \A(E)$ such that $\overline{\im T'} = F$.
\end{lemma}

\begin{proof}
We may suppose that $E$ is infinite-dimensional, since otherwise the lemma follows from routine linear algebra. Let $\{ \lambda_n : n \in \N \}$ be a dense subset of the unit ball of $F$, and let $(b_n)$ be a normalised basic sequence in $E$. Let $(\beta_n) \subset E'$ satisfy $\langle b_i, \beta_j \rangle = \delta_{ij} \ (i, j \in \N)$. Define $T = \sum_{n=1}^\infty 2^{-n} b_n \otimes \lambda_n$. The operator $T$ is a limit of finite-rank operators and 
$$T'\varphi = \sum_{n=1}^\infty 2^{-n} \varphi(b_n) \lambda_n \quad (\varphi \in E').$$
 Certainly $\overline{\im T'} \subset F$. Observing that $T'(2^i \beta_i) = \lambda_i \ (i \in \N)$, we see that $\overline{\im T'} = F$, as required.
\end{proof}

We can now give our characterisation of topological left Noetherianity for $\A(E)$. We notice that our proof actually implies that for these Banach algebras topological left Noetherianity is equivalent
to every closed left ideal being topologically principal.

\begin{theorem}	\label{2.3.8}
Let $E$ be a Banach space such that $\A(E)$ has a left approximate identity. Then the following are equivalent:
	\begin{enumerate}
		\item[{\rm (a)}] the Banach algebra $\A(E)$ is topologically left Noetherian;
		\item[{\rm (b)}] every closed left ideal of $\A(E)$ is topologically principal;
		\item[{\rm (c)}] the space $E'$ is separable.
	\end{enumerate}
\end{theorem}

\begin{proof}
It is trivial that (b) implies (a). To see that (c) implies (b), note that, by Theorem \ref{2.3.4}, every closed left ideal of $\A(E)$ has the form $\mathscr{L}(F)$, for some closed linear subspace $F$ in $E'$. Fixing $F \in \SUB(E')$, by Lemma \ref{2.3.7} there exists $T \in \A(E)$ such that $F = \overline{ \im T'}$, which implies that $\mathscr{L}(F) = \overline{\A(E) T}$, by Lemma \ref{2.3.6}(i).

We show that (a) implies (c) to complete the proof. Suppose that $\A(E)$ is topologically left Noetherian. Then in particular 
$\A(E) =\overline{\A(E)T_1+ \cdots + \A(E)T_n}$ for some $T_1, \ldots, T_n \in \A(E)$. Observing that $\mathscr{L}(E') = \A(E)$, Lemma \ref{2.3.5} implies 
that $E' = \overline{\im T_1' + \cdots + \im T_n'}$. Since each operator $T_i$ is compact, so is each $T_i'$, implying that each space $\im T_i'$ is separable. 
It follows that $E' =  \overline{\im T_1' + \cdots + \im T_n'}$ is separable.
\end{proof}

We now give our classification of the closed right ideals of $\A(E)$. Observe that our hypothesis on $\A(E)$ changes from possessing a left approximate identity to possessing a right approximate identity.

\begin{theorem}	\label{2.3.9}
Let $E$ be a Banach space such that $\A(E)$ has a right approximate identity. 
There is a lattice isomorphism $\Xi : (\SUB(E), \subset) \rightarrow (\RID(E), \subset)$ given by
$$\Xi : F \mapsto \mathscr{R}(F),$$
with inverse given by
$$\widehat{\Xi} : I \mapsto \overline{\spn}_{T \in I} ( \im T ) \quad (I \in \RID(\A(E))).$$
\end{theorem}

\begin{proof}
It is clear that $\Xi$ and $\widehat{\Xi}$ are inclusion preserving. Since a poset isomorphism between lattices preserves the lattice structure, once we have shown that $\Xi$ and $\widehat{\Xi}$ are mutually inverse it will follow that they are lattice isomorphisms.

Let $F$ be a closed linear subspace of $E$ and set $D= \widehat{\Xi}(\mathscr{R}(F))$. It is immediate from the definitions that $D \subset F$. Moreover, given $x \in F$, by considering $x \otimes \lambda$ for some $\lambda \in E \setminus \{ 0 \}$ we see that $x \in D$. Hence $F = D$, and, since $F$ was arbitrary, this shows that $\widehat{\Xi} \circ \Xi$ is the identity map. 

Let $I$ be a closed right ideal of $\A(E)$, and set $F = \widehat{\Xi}(I)$. It is clear that $I \subset \mathscr{R}(F)$. By Lemma \ref{2.1.1} the finite-rank operators intersect $\mathscr{R}(F)$ densely, so in order to check the reverse inclusion it is sufficient to show that $\mathcal{F}(E) \cap \mathscr{R}(F) \subset I$. Let $T \in \mathcal{F}(E) \cap \mathscr{R}(F).$ Then we can write $T = \sum_{i=1}^n x_i \otimes \lambda_i$, for some $n \in \N$, some $x_1, \ldots, x_n \in \im T$, and some $\lambda_1, \ldots, \lambda_n \in E'$. Fix $i \in \{1, \ldots, n \}$. Then $x_i \in F$ so there exists a sequence $(y_j) \subset \spn_{U \in I} (\im U)$ such that $\lim_{j \rightarrow \infty} y_j = x_i$. Moreover, for each $j$ we can write $y_j = S_1^{(j)} z_1 + \cdots + S_{k_j}^{(j)}z_{k_j}$, for some $k_j \in \N$, some $S_1^{(j)}, \ldots, S_{k_j}^{(j)} \in I,$ and some $z_1, \ldots, z_{k_j} \in E$. For each $j$, and each $p=1, \ldots, k_j$ we have 
$\left(S_p^{(j)}z_j \right) \otimes \lambda_i = S_p^{(j)} \circ (z_p \otimes \lambda_i)  \in I$. Hence $y_j \otimes \lambda_i \in I$ for each $j$, so that, taking the limit as $j$ goes to infinity, $x_i \otimes \lambda_i \in I$. As $i$ was arbitrary it follows that $T \in I$. Hence we have shown that $I = \mathscr{R}(F)$. As $I$ was arbitrary, we have shown that $\Xi \circ \widehat{\Xi}$ is the identity map. 
\end{proof}

Now we set out to characterise when $\A(E)$ is topologically right Noetherian, for $E$ a Banach space as in Theorem \ref{2.3.9}.

\begin{lemma}	\label{2.3.10}
Let $E$ and $\widehat{\Xi}$ be as in Theorem \ref{2.3.9}. Let $T_1, \ldots, T_n \in \A(E)$ and let $I = \overline{T_1\A(E)+ \cdots +T_n\A(E)}$. Then
$$\widehat{\Xi}(I) = \overline{ \im T_1 + \cdots + \im T_n}.$$
\end{lemma}

\begin{proof}
Since each $T_i \ (i = 1, \ldots, n)$ belongs to $I$ we have $\overline{\im T_1 + \cdots + \im T_n} \subset \widehat{\Xi}(I)$. Let $x \in \widehat{\Xi}(I)$, and let $\eps>0$. Then, by the definition of $\widehat{\Xi}$, there exist $m \in \N$, $S_1, \ldots, S_m \in I$, and $y_1, \ldots, y_m \in E$ such that
$$\Vert x - (S_1y_1 + \cdots + S_my_m) \Vert < \eps.$$ 
Since $T_1\A(E) + \cdots + T_n\A(E)$ is dense in $I$, we may in fact suppose that
 $$S_1, \ldots, S_m \in T_1\A(E)+\cdots + T_n\A(E),$$
 so that $S_1y_1 + \cdots + S_my_m \in \im T_1 + \cdots + \im T_n$. 
As $\eps$ was arbitrary we see that $x \in \overline{\im T_1 + \cdots + \im T_n}$. The result now follows.
\end{proof}

We omit the proof of the following well known result. In any case, it can be proved in a similar fashion to Lemma \ref{2.3.7}.

\begin{lemma} \label{2.3.11}
Let $E$ be a Banach space, and let $F$ be any separable Banach space. Then there exists an approximable linear map from $E$ to $F$ with dense range.
\end{lemma}


We can now prove the theorem.

\begin{theorem}	\label{2.3.12}
Let $E$ be a Banach space such that $\A(E)$ has a right approximate identity. Then the following are equivalent:
	\begin{enumerate}
		\item[{\rm (a)}]  the Banach algebra $\A(E)$ is topologically right Noetherian;
		\item[{\rm (b)}]  every closed right ideal of $\A(E)$  is topologically principal;
		\item[{\rm (c)}]	the space $E$ is separable.
	\end{enumerate}
\end{theorem}

\begin{proof}
It is trivial that (b) implies (a). We show that (a) implies (c). Suppose that $\A(E)$ is topologically right Noetherian. Then $\A(E) = \mathscr{R}(E)$ is topologically finitely-generated so that, 
by Lemma \ref{2.3.10}, there exist $n \in \N$ and $T_1, \ldots, T_n \in \A(E)$ such that $E = \overline{\im T_1 + \cdots + \im T_n}.$
 Since each operator $T_i \ (i=1, \ldots, n)$ is compact, its image is separable, and hence so is $E$.

Now suppose instead that $E$ is separable, and let $I$ be a closed right ideal of $\A(E)$. 
Then, by Theorem \ref{2.3.9}, $I = \mathscr{R}(F)$ for 
some $F \in \SUB(E).$ By Lemma \ref{2.3.11} there exists $T \in \A(E)$ with $\overline{\im T} = F$. By Lemma \ref{2.3.10} we have $\widehat{\Xi}\left(\overline{T\A(E)} \right) = \overline{\im T} = F$,
so that, by Theorem \ref{2.3.9}, $I = \mathscr{R}(F) = \overline{T\A(E)}.$ Since  $I$ was arbitrary, this shows that (c) implies (b).
\end{proof}

We can now prove Theorem \ref{1.3} as a special case of our results above.

\begin{proof}[Proof of Theorem \ref{1.3}]
  Of course, under either hypothesis $\K(E) = \A(E)$. By \cite[Theorem 2.5 (ii)]{D}, $\K(E)$ has a left approximate identity whenever $E$ has the approximation property. Similarly,
  by \cite[Theorem 3.3]{GW} $\K(E)$ has a (bounded) two-sided approximate identity whenever $E'$ has the bounded approximation property. Hence the results follow from Theorem \ref{2.3.8} and
  Theorem \ref{2.3.12}.
\end{proof}

\textit{Remark.} Consider $\K(\ell^{\, 1})$. Of course, $(\ell^{\, 1})' \cong \ell^{\, \infty}$, which has BAP by \cite[Example 5(a), Chapter II E]{Woj}. Hence, by Theorem \ref{1.3}, $\K(\ell^{\, 1})$ is an
example of a Banach algebra which is topologically right Noetherian, but not topologically left Noetherian.

\vskip 2mm

We observe that, although it talks about algebraically finitely-generated ideals, the argument given in \cite[Corollary 3.2]{DKKKL} actually proves that any Banach space $E$ satisfying its hypothesis has
the property that $\B(E)$ is not topologically left Noetherian. Indeed, the argument there is to demonstrate that there are more (maximal) closed left ideals in $\B(E)$ than there are finite $n$-tuples of operators. Hence not every closed left ideal can be topologically finitely-generated. This covers a large class of Banach spaces, including, for example, $c_0$, $\ell^{\, p}$ for $1 \leq p < \infty$, $L^{\, p}[0,1]$ for $1<p<\infty$,
and many other spaces discussed in \cite{DKKKL}. In the case that $E$ is a reflexive Banach space, $\B(E)$ is a dual Banach algebra with predual given by $E \widehat{\otimes} E'$. We shall show in the next
section that, for a reflexive Banach space $E$ with the approximation property, $\B(E)$ is weak*-topologically left Noetherian whenever $\K(E)$ is. Hence in particular, many of the above examples which fail to be $\Vert \cdot \Vert$-topologically left Noetherian are weak*-topologically left Noetherian.

We note however that it is possible for $\B(E)$ to be topologically left Noetherian, for an infinite-dimensional Banach space $E$.
Let $E_{AH}$ be the Banach space constructed by Argyros and Haydon in \cite{AH} with the property that 
$\B(E_{AH}) = \C {\rm \, id}_{E_{AH}} + \K(E_{AH}).$
Since $E_{AH}$ is a predual of $\ell^{\, 1}$, which has BAP, it satisfies the hypotheses of part (i) and (ii) of  Theorem \ref{1.3}.
Since $\B(E_{AH}) = \K(E_{AH})^\sharp$, Theorem \ref{2.3.8} and Lemma \ref{2.1.1a}(iii) imply
that $\B(E_{AH})$ is $\Vert \cdot \Vert$-topologically left and right Noetherian. 
Further examples come from \cite{MPZ}, where  the authors construct, for each countably infinite, compact metric space $X$, a predual of  $\ell^{\, 1}$, say $E_X$, such that $\B(E_X)/\K(E_X) \cong C(X)$. Since $C(X)$ is topologically Noetherian for any compact metric space $X$, a similar argument shows that the Banach algebras $\B(E_X)$ are topologically left and right Noetherian. In all of these examples the Banach space is hereditarily indecomposable. There is no hereditarily indecomposable Banach space $E$ for which we know that $\B(E)$ is not topologically left/right Noetherian.



\section{Multiplier Algebras and Dual Banach Algebras}		\label{S3.4}
\noindent
In this section we consider those Banach algebras $A$ whose multiplier algebra is a dual Banach algebra. We shall focus on the case in which $A$ has a bounded approximate identity. Examples of such Banach algebras included $L^1(G)$ for $G$ a locally compact group, and $\B(E)$ for $E$ a reflexive Banach space with AP. Other examples include the Fig{\`a}-Talamanca--Herz algebras $A_p(G)$ for $G$ a locally compact amenable group, and $p \in (1, \infty)$, with the predual of $M(A_p(G))$ given by $PF_p(G)$, the algebra of p-pseudo-functions of $G$. Also $L^{\, 1}(\mathbb{G})$, where $\mathbb{G}$ is a locally compact quantum group in the sense of Kustermans and Vaes, fits into this setting whenever it has a bounded approximate identity \cite{Da}.

We prove in Proposition \ref{2.0.9} that, for a Banach algebra $A$ satisfying a fairly mild condition, the multiplier algebra $M(A)$ is weak*-topologically left Noetherian
whenever $A$ is $\Vert \cdot \Vert$-topologically left Noetherian. Corollary \ref{1.4} then follows. We go on to prove that for a certain, more restrictive class of Banach algebras there is a bijective
correspondence between the closed left ideals of $A$ and the weak*-closed left ideals of $M(A)$ (Theorem  \ref{2.0.11b}). 

For background on multiplier algebras see one of \cite{Dal, Da, P}. We say that $A$ is \textit{faithful} if $Ax = \{ 0 \}$ implies $x = 0 \ (x \in A)$, and also $xA = \{ 0 \}$ implies $x = 0 \ (x \in A)$. We recall that the canonical map
 $A \rightarrow M(A)$ is injective if and only if $A$ is faithful. When $A$ has a bounded approximate identity, this map is bounded below, so that $A$ is isomorphic to its image inside $M(A)$. When this is the case we shall identify $A$ with its image inside $M(A)$.

In this section, when we consider a linear functional applied to a vector, we shall often use a subscript to indicate the exact dual pairing. So for example, if $A$ is a Banach algebra, and $a \in A$ and $f \in A'$, we might write $\langle a, f \rangle_{(A, \, A')}$ or $\langle f, a \rangle_{(A', \, A)}$ for the value of $f$ applied to $a$.

In \cite{Da} Daws considers Banach algebras whose multiplier algebras are also dual Banach algebras. We shall use the following consequence of Daws' work, which essentially says that when such a Banach algebra has a bounded approximate identity, the multiplier and dual structures are compatible in a natural way.

\begin{theorem}		\label{5.1}
Let $A$ be a Banach algebra with a bounded approximate identity, and suppose that $M(A)$ is a dual Banach algebra, with predual $X$. Then $X$ may be identified with a closed $A$-submodule of  $A \cdot A' \cdot A$ in such a way that 
	\begin{equation}		\label{eq5.1} 
		\langle f \cdot a, \mu \rangle_{(X, \, M(A))} = \langle f, a\mu \rangle_{(A', \, A)},
	\end{equation}
	\begin{equation}		\label{eq5.2} 
		\langle a \cdot f, \mu \rangle_{(X, \, M(A))} = \langle f, \mu a \rangle_{(A', \, A)},
	\end{equation}
for all $\mu \in M(A)$, and  all $a \in A$ and $f \in A'$ with $f\cdot a \in X$/ $a \cdot f \in X$ respectively. We also have
	\begin{equation}		\label{eq5.3}
		\langle x, a \rangle_{(X, \, M(A))} = \langle x, a \rangle_{(A', \, A)},
	\end{equation}
for all $x \in X$ and $a \in A$.
\end{theorem}

\begin{proof}
The fact that $X$ may be identified with a closed $A$-submodule of  $A \cdot A' \cdot A$ follows immediately from \cite[Theorem 7.9]{Da} and the remarks following it, and Equations \eqref{eq5.1} and \eqref{eq5.2} then follow by chasing through the definition of the map $\theta_0$ of that theorem.

Equation \eqref{eq5.3} then follows from \eqref{eq5.1}: given $a \in A$ and $x \in X$, let $b \in A$ and $f \in A'$ satisfy $x = f \cdot b$. Then 
	\begin{align*}
		\langle f \cdot b, a \rangle_{(X, \, M(A))} 
		= \langle f, ba \rangle_{(A', \, A)}
		=  \langle f \cdot b, a \rangle_{(A', \, A)},
	\end{align*}
	as required.
\end{proof}

\textit{Remark.} Commutative Banach algebras whose multiplier algebras are dual Banach algebras satisfying \eqref{eq5.1}/\eqref{eq5.2} were considered by \"Ulger in \cite{U}. Since \"Ulger always assumes the existence of a bounded approximate identity, Theorem \ref{5.1} allows the hypothesis of \cite[Theorem 3.7]{U} to be simplified slightly.
\vskip 2mm

We note the following.

\begin{lemma} \label{2.0.5}
Let $A$ be a Banach algebra with a bounded approximate identity, such that $M(A)$ is a dual Banach algebra. Then $A$ is weak*-dense in $M(A)$.
\end{lemma}

\begin{proof}
 Let $X$ be the predual of $M(A)$.
  Suppose $x \in A_{\perp} \subset X$. Then, by Theorem \ref{5.1}, we may identify $X$ with a closed subspace of $A'$, and for all $a \in A$ we have $0 = \langle x, a \rangle_{(X, \, M(A))} = \langle x, a \rangle_{(A', \, A)}$, which implies that $x = 0$. Hence $\overline{A}^{w^*} = (A_\perp)^\perp = \{ 0 \}^\perp = M(A)$.
\end{proof}

We now prove a result about weak*-topological left Noetherianity in this setting. Note that,  by the previous lemma,  the hypothesis is satisfied by any Banach algebra with a bounded approximate identity whose multiplier algebra is a dual Banach algebra. 

\begin{proposition}  \label{2.0.9}
Let $A$ be a Banach algebra such that $M(A)$ admits the structure of a dual Banach algebra in such a way that $A$ is weak*-dense in $M(A)$. Suppose that for every closed left ideal $I$ in $A$ there exists $n \in \N$ and there exist $\mu_1, \ldots, \mu_n \in M(A)$ such that $I = \overline{A^\sharp\mu_1+\cdots + A^\sharp \mu_n}$. Then $M(A)$ is weak*-topologically left Noetherian. In particular, $M(A)$ is weak*-topologically left Noetherian whenever $A$ is $\Vert \cdot \Vert$-topologically left Noetherian.
\end{proposition}

\begin{proof}
Let $I$ be a weak*-closed left ideal of $M(A)$. Since $A$ is weak*-dense in $M(A)$, which is unital, Lemma \ref{2.1.1} implies that $A \cap I$ is weak*-dense in $I$. 
On the other hand, $A \cap I$ is a closed left ideal in  $A$, so there exists $n \in \N$, and there exist $\mu_1, \ldots, \mu_n \in M(A)$ such that $A \cap I = \overline{A^\sharp \mu_1+\cdots + A^\sharp \mu_n}$. It follows that 
$$I = \overline{A^\sharp \mu_1+\cdots + A^\sharp \mu_n}^{w^*} = \overline{M(A) \mu_1+\cdots + M(A) \mu_n}^{w^*}.$$
 As $I$ was arbitrary the result follows.
\end{proof}

We are now able to prove Corollary \ref{1.4} concerning the weak*-topological left/right Noetherianity for algebras of the form $M(G)$ and $\B(E)$. Note that, by Proposition \ref{2.1.3} and the discussion at the end of
Section 4, these algebras are often not $\Vert \cdot \Vert$-topologically left/right Noetherian 

\begin{proof}[Proof of Corollary \ref{1.4}]
  This follows from Proposition \ref{2.0.9}, Theorem \ref{1.2}, and Theorem \ref{1.3}.
\end{proof}

\begin{definition} \label{2.0.6}
Let $A$ be a Banach algebra. We say that $A$ is \textit{a compliant Banach algebra} if $A$ is faithful and $M(A)$ is a dual Banach algebra in such a way that, for each $a \in A$, the maps $M(A) \rightarrow A$ given by $\mu \mapsto \mu a$ and $\mu \mapsto a \mu$ are weak*-weakly continuous.
\end{definition}

In this article we shall consider the  ideal structure of compliant Banach algebras, but we note that they appear to have interesting properties more broadly and are worthy of further study. In the papers \cite{HA1} and \cite{HA2}  Hayati and Amini consider Connes amenability of certain multiplier algebras which are also dual Banach algebras. In our terminology, \cite[Theorem 3.3]{HA2} says that if $A$ is a compliant Banach algebra with a bounded approximate identity, then $A$ is amenable if and only if $M(A)$ is Connes amenable.

We have the following family of examples of compliant Banach algebras.

\begin{lemma}		\label{2.0.4a}
Let $A$ be a Banach algebra with a bounded approximate identity which is Arens regular and an ideal in its bidual. Then $A$ is a compliant Banach algebra.
\end{lemma}

\begin{proof}	
By \cite[Theorem 3.9]{L} $A''$ as an algebra with Arens multiplication may be identified
 with $M(A)$. Arens regularity implies that $A''$ is a dual Banach algebra with predual
  $A'$. Given $a \in A$, the maps $\mu \mapsto \mu a$ and $\mu \mapsto a \mu$ are 
  weak*-continuous as maps from $A''$ to itself, and hence they are weak*-weakly
  continuous when considered as maps from $A''$ to $A$. 
\end{proof}

It follows from Lemma \ref{2.0.4a} that $c_0(\N)$ is an example of a compliant Banach algebra. A family of examples that will be important to us is the following:

\begin{corollary}		\label{2.0.4b}
Let $E$ be a reflexive Banach space with the approximation property. Then $\K(E)$ is a compliant Banach algebra.
\end{corollary}

\begin{proof}
  By \cite[Theorem 3]{Young} $\K(E)$ is Arens regular. Moreover $\K(E)'' = \B(E)$, so that we see that $\K(E)$ is an ideal in its bidual. Hence the result follows from
  the previous lemma.
\end{proof}

The following lemma is also useful for finding examples.

\begin{lemma}		\label{5.2}
Let $A$ be as in Theorem \ref{5.1}, and suppose that the identification of that theorem yields $X = A \cdot A' = A' \cdot A$. Then $A$ is a compliant Banach algebra.
\end{lemma}

\begin{proof}
Let $(\mu_\alpha)$ be a net in $M(A)$ which converges to some $\mu \in M(A)$ in the weak*-topology. Fix $a \in A$, and let $f \in A'$ be arbitrary. Then
$$\lim_\alpha \langle f, a\mu_\alpha \rangle_{(A', \, A)} 
= \lim_\alpha \langle f \cdot a, \mu_\alpha \rangle_{(X, \, M(A))} 
= \langle f \cdot a, \mu \rangle_{(X, \, M(A))} 
=\langle f, a\mu \rangle_{(A', \, A)}.$$
This shows that the map $\mu \mapsto a\mu$ is weak*-weakly continuous, and by an analogous argument so is the map $\mu \mapsto \mu a$. As $a$ was arbitrary this proves the lemma.
\end{proof}

Compliance is a fairly restrictive condition, as the next Proposition illustrates.

\begin{proposition} \label{2.0.8}
	\begin{enumerate}
		\item[{\rm (i)}] Let $A$ be a compliant Banach algebra, and let $\Box$ denote the first Arens product on $A''$. Then $A$ is an ideal in $(A'', \, \Box)$.
		\item[{\rm (ii)}] Let $K$ be a locally compact space. The Banach algebra $C_0(K)$ is compliant if and only if $K$ is discrete.
		\item[{\rm (iii)}] Let $G$ be a locally compact group. The Banach algebra $L^{\, 1}(G)$ is compliant if and only if $G$ is compact.
	\end{enumerate}
\end{proposition}

\begin{proof}
(i) If $A$ is compliant then, for every $a \in A$, the maps given by  
$$L_a \colon b \mapsto ab \qquad \text{and} \qquad R_a \colon b \mapsto ba$$
are weakly compact. Hence $L_a''(A''), R_a''(A'') \subset A$ by \cite[Theorem 3.5.8]{Meg}.
 Given $a \in A$ we have $L_a''  \colon \Psi \mapsto a \Box \Psi \ (\Psi \in A''),$  and we see that $A$ is a right ideal in $A''$. Similarly it is a left ideal.
 
(ii) Whenever $K$ is discrete $C_0(K)$ is Arens regular with bidual given by $\ell^{\, \infty}(K)$. Hence the algebra is compliant by Lemma \ref{5.2}. 

Now suppose instead that $C_0(K)$ is compliant. The dual space of $C_0(K)$ may be identified with $M(K)$, and, as such, for each $x \in K$, we may define an element $\eps_x \in C_0(K)''$ by 
$$\eps_x \colon \mu \mapsto \int_{\{x \}} 1 \dd \mu \quad (\mu \in M(K)).$$ 
By part (i) we know that, for any $x \in K$ and any $f \in C_0(K)$, we have $\eps_x \Box f \in C_0(K)$. We then calculate that, for $\mu \in M(K)$, we have 
$$\langle f \Box \eps_x, \mu \rangle = \langle \eps_x, \mu \cdot f \rangle = \int_{\{ x \}} f \dd \mu = \mu(\{ x \})f(x)$$
so that $\eps_x \Box f$ is equal to $f(x)$ at $x$, and $0$ everywhere else. Therefore, given any $x \in K$, by choosing any $f \in C_0(K)$ not vanishing at $x$, we see that the point mass at $x$ is continuous. Hence $K$ is discrete.

(iii) If $L^{\, 1}(G)$ is compliant, then, by part (i) and \cite{Gr}, $G$ is compact.  Suppose instead that $G$ is compact. By examining the maps in \cite[Theorem 7.9]{Da} we see that the identification in Theorem \ref{5.1} is the usual inclusion $C(G) \rightarrow L^{\infty}(G)$. Then $C(G) = L^{\, 1}(G) \cdot L^{\, \infty}(G) = 
L^{\, \infty}(G) \cdot L^{\, 1}(G)$ by \cite[Proposition 2.39(d)]{F}. Hence, by Lemma \ref{5.2}, $L^{\, 1}(G)$ is compliant.
\end{proof}

For compliant Banach algebras there is a bijective correspondence between the closed left ideals of $A$ and the weak*-closed left ideals of $M(A)$ as we describe below in Theorem \ref{2.0.11b}. The next section will be devoted to applications of this result.

\begin{lemma} \label{2.0.11a}
Let $I$ be a closed left ideal of a compliant Banach algebra $A$, and let $\mu \in \overline{I}^{w^*} \subset M(A)$. Then $A \mu \subset I$.
\end{lemma}

\begin{proof}
Let $(\mu_\alpha)$ be  a net in $I$ converging to $\mu$ in the weak*-topology and let $a \in A$. For each index $\alpha$ we have $a \mu_\alpha \in I$. Since $A$ is compliant, the net $a \mu_\alpha$ converges weakly to $a \mu$ in $A$. Hence $a \mu \in \overline{I}^w = I$. As $a$ was arbitrary, the result follows.
\end{proof}

\begin{theorem} \label{2.0.11b}
Let $A$ be a compliant Banach algebra with a bounded approximate identity. The map
$$I \mapsto \overline{I}^{w^*},$$
defines a bijective correspondence between closed left ideals in $A$ and weak*-closed left ideals in $M(A)$. The inverse is given by
$$J \mapsto A \cap J,$$
for $J$ a weak*-closed left ideal in $M(A)$.
\end{theorem}

\begin{proof}
First we take an arbitrary closed left ideal $I$ in $A$ and show that $A \cap \overline{I}^{w^*} = I$. Certainly $I \subset A \cap \overline{I}^{w^*}$. Let $a \in A \cap \overline{I}^{w^*}$. Then by Lemma \ref{2.0.11a} we have $Aa \subset I$. Since $A$ has a bounded approximate identity, this implies that $a \in I$. As $a$ was arbitrary, we must have $I = A \cap \overline{I}^{w^*}$.

It remains to show that, given a weak*-closed left ideal $J$ of $M(A)$, we have $\overline{A \cap J}^{w^*} = J$, and this follows from Lemma \ref{2.1.1} and Lemma \ref{2.0.5}.
\end{proof}

Finally we show that for compliant Banach algebras the converse of Proposition \ref{2.0.9} holds, so that weak*-topological left Noetherianity of $M(A)$ can be characterised in terms of a $\Vert \cdot \Vert$-topological condition on $A$.

\begin{proposition} \label{2.0.11}
Let $A$ be a compliant Banach algebra with a bounded approximate identity. Then $M(A)$ is weak*-top\-o\-l\-og\-i\-ca\-lly left Noetherian if and only if every closed left ideal $I$ in $A$ has the  form 
$$I = \overline{A\mu_1+\cdots + A\mu_n},$$
 for some $n \in \N$, and some $\mu_1, \ldots, \mu_n \in M(A)$.
\end{proposition}

\begin{proof}
The ``if'' direction follows from Proposition \ref{2.0.9} and Lemma \ref{2.0.5}. Conversely, suppose that $M(A)$ is weak*-topologically left Noetherian, and let $I$ be a closed left ideal of $A$. Then there exist $n \in \N$ and $\mu_1, \ldots, \mu_n \in M(A)$ such that 
$$\overline{I}^{w^*} = \overline{M(A)\mu_1+\cdots+M(A)\mu_n}^{w^*} = \overline{A\mu_1+\cdots+A\mu_n}^{w^*},$$
where we have used Lemma \ref{2.0.5} to get the second equality. Hence, by applying Theorem \ref{2.0.11b} twice, we obtain
$$I = \overline{I}^{w^*} \cap A = \overline{A\mu_1+\cdots+A\mu_n}^{w^*} \cap A =  \overline{A\mu_1+\cdots+A\mu_n}.$$
The result follows.
\end{proof}

\section{Some Classification Results}
\noindent
In this section we use Theorem \ref{2.0.11b} to give classifications of the weak*-closed left ideals of $M(G)$, for $G$ a compact group, and of the weak*-closed left ideals of $\B(E)$, for $E$ a reflexive Banach space with the approximation property. We then observe how this gives us some classification results for the closed right submodules of the preduals.

Let $G$ be a compact group and suppose the for each  $\pi \in \widehat{G}$ we have chosen a linear subspace $E_\pi \leq H_\pi$. Then we define
$$J[(E_\pi)_{\pi \in \widehat{G}}] := \left\{\mu \in M(G) : \pi(\mu)(E_\pi) = 0, \, \pi \in \widehat{G} \right\}.$$
We shall show that these are exactly the  weak*-closed left ideals of $M(G)$.

\begin{lemma}		\label{2.0.12}
Let $G$ be  a compact group and let $E_\pi \leq H_\pi \ (\pi \in \widehat{G})$. Then
	\begin{equation}		
	\overline{\spn} \left\{ \xi *_\pi \eta : \pi \in \widehat{G}, \, \xi \in E_\pi, \, \eta \in H_\pi \right\}^\perp
	= J[(E_\pi)_{\pi \in \widehat{G}}].
	\end{equation}
\end{lemma}

\begin{proof}
Routine calculation.
\end{proof}

\begin{theorem}		\label{2.0.13}
Let $G$ be a compact group. Then the  weak*-closed left ideals of $M(G)$ are given by $J[(E_\pi)_{\pi \in \widehat{G}}]$, as $(E_\pi)_{\pi \in \widehat{G}}$ runs over the  possible choices of linear subspaces $E_\pi \leq H_\pi \ (\pi \in \widehat{G})$. 
\end{theorem}

\begin{proof}
By Lemma \ref{2.0.12} each space $J[(E_\pi)_{\pi \in \widehat{G}}]$ is weak*-closed, and it is easily checked that it is a left ideal. Moreover, by Theorem \ref{2.1.2a} each closed left ideal of $L^{\, 1}(G)$ has the form 
 $L^{\, 1}(G) \cap J[(E_\pi)_{\pi \in \widehat{G}}]$, for some choice of subspaces 
 $E_\pi \leq H_\pi \ (\pi \in \widehat{G}).$
By Proposition \ref{2.0.8}, $L^{\, 1}(G)$ is a compliant Banach algebra, so we may apply Proposition \ref{2.0.11b} to see that this must be the full set of weak*-closed left ideals.
\end{proof}

Recall the definition of $\mathscr{I}_{\B(E)}(F)$ given in Equation \eqref{eq2.4.0c}.

\begin{theorem}		\label{2.5.5a}
Let $E$ be a reflexive Banach space with the approximation property. Then the weak*-closed left ideals of $\B(E)$ are exactly given by $\mathscr{I}_{\B(E)}(F)$, as $F$ runs through $\SUB(E)$. The weak*-closed right ideals are given by $\mathscr{R}_{\B(E)}(F)$, as $F$ runs through $\SUB(E)$.
\end{theorem}	

\begin{proof}
For any closed linear subspace $F \subset E$ the left ideal $\mathscr{I}_{\B(E)}(F)$ is 
weak*-closed since we have
$$ \mathscr{I}_{\B(E)}(F) = \{ x \otimes \lambda : x \in F, \, \lambda \in E' \}^\perp,$$
where $x \otimes \lambda$ denotes an element of the predual $E \widehat{\otimes} E'$. Similarly we have
$$\mathscr{R}_{\B(E)}(F) = \{x \otimes \lambda : x \in E, \, \lambda \in F^{\perp} \}^\perp,$$
so that these right ideals are weak*-closed.
Observe that for $F \in \SUB(E)$, we have 
\begin{align}		\label{eq6.1}
\mathscr{L}_{\B(E)}(F^\perp) &= \{ T \in \B(E) : \im T' \subset F^\perp \} 
= \{ T \in \B(E) : \overline{\im T'} \subset F^\perp \} \\
&= \{ T \in \B(E) : \left(\overline{\im T'}\right)_\perp \supset F \} 
= \{ T \in \B(E) : \ker T \supset F \} = \mathscr{I}_{\B(E)}(F), \notag
\end{align}
so that, by Theorem \ref{2.3.4}, every closed left ideal of $\K(E)$ has the form $\mathscr{I}_{\B(E)}(F) \cap \K(E)$ for some $F \in \SUB(E)$. By Corollary \ref{2.0.4b} $\K(E)$ is compliant, so we may apply Proposition \ref{2.0.11b} to see that every weak*-closed left ideal has the required form. A similar argument applies to the weak*-closed right ideals.
\end{proof}

Finally we show how, using the following proposition, we can describe the closed left/right submodules of $E\widehat{\otimes} E'$, for $E$ a reflexive Banach space with AP. We also obtain a description of the left-translation-invariant closed subspaces of $C(G)$, for $G$ a compact group (compare with \cite[Theorem 2]{Ak}). 

\begin{proposition}		\label{6.1}
Let $(A, X)$ be a dual Banach algebra. Then there is a bijective correspondence between the closed right $A$-submodules of $X$ and the weak*-closed left ideals of $A$ given by 
$$Y \mapsto Y^\perp,$$
for $Y$ a closed right $A$-submodule of $X$.
\end{proposition}

\begin{proof}
It is quickly checked that $Y^\perp$ is a left ideal, whenever $Y$ is a right submodule, and that $I_\perp$ is a right submodule whenever $I$ is a left ideal. Equation \eqref{eq2.1} now tells us that the given correspondence is bijective, with inverse given by $I \mapsto I_\perp$,  for $I$ a weak*-closed left ideal.
\end{proof}

In the following corollary, given a Banach space $E$, and a closed subspaces $F \subset E$ and $D \subset E'$, we identify $F \widehat{\otimes} E'$ with the closure of the algebraic tensor product $F \otimes E'$ inside $E \widehat{\otimes} E'$, and similarly for $E \widehat{\otimes} D$.

\begin{corollary}		\label{6.2}
Let $E$ be a reflexive Banach space with the approximation property. Then the closed right $\B(E)$-submodules of $E \widehat{\otimes} E'$ are given by 
$$F \widehat{\otimes} E' \quad (F \in \SUB(E)).$$
The closed left $\B(E)$-submodules are given by 
$$E \widehat{\otimes} D  \quad (D \in \SUB(E')).$$
\end{corollary}

\begin{proof}
Given $F \in \SUB(E)$ it is easily seen that $F \widehat{\otimes} E'$ is a closed right
 submodule. Hence $I := (F \widehat{\otimes} E')^\perp$ is a weak*-closed left ideal by
  Proposition \ref{6.1}. It is easily checked that $E' \circ I = F^\perp$, and hence 
  $I = \mathscr{L}_{\B(E)}(F^\perp) = \mathscr{I}_{\B(E)}(F)$ by \eqref{eq6.1}. Since the correspondence given in Proposition \ref{6.1} is bijective, and since by Theorem \ref{2.5.5a} every weak*-closed ideal has the form $\mathscr{I}_{\B(E)}(F)$ for some $F$, it must be that every closed right submodule has the form $F \widehat{\otimes} E'$ for some $F \in \SUB(E)$. 
  
  The result about closed left submodules is proved analogously.
\end{proof}

\begin{corollary}		\label{2.0.14}
Let $G$ be a compact group, and let $X \subset C(G)$ be a closed linear subspace, which is invariant under left translation. Then there exists a choice of  
linear subspaces $E_\pi \leq H_\pi \ (\pi \in \widehat{G})$ such that 
$$X = \overline{\spn} \left\{ \xi *_\pi \eta : \pi \in \widehat{G}, \, \xi \in E_\pi, \, \eta \in H_\pi \right\}.$$
\end{corollary}

\begin{proof}
In fact, by the weak*-density of the discrete measures in $M(G)$, the closed right submodules of $C(G)$ coincide with the closed linear subspaces invariant under left translation (compare with \cite[Lemma 3.3]{W1}). By Proposition \ref{6.1} $X$ has the form $I_\perp$, for some weak*-closed left ideal $I$ of $M(G)$. It now follows from Theorem \ref{2.0.13} and Lemma \ref{2.0.12} that $X$ has the given form.
\end{proof}

\subsection*{Acknowledgements}
\noindent
This work was supported by the French
``Investissements d'Avenir'' program, project ISITE-BFC (contract
ANR-15-IDEX-03). The article is based on part of the author's PhD thesis, and as such he would like to thank his doctoral supervisors Garth Dales and Niels Laustsen, as well as his examiners Gordon Blower and Tom K\"orner, for their careful reading of earlier versions of this material and their helpful comments. We would also like to thank Yemon Choi for some helpful email exchanges. Finally, we would like to thank the anonymous referee for his/her many insightful comments, and in particular for conjecturing Proposition \ref{2.0.8}(ii).


\begin{thebibliography}{99}

\bibitem{Ak} C.\ A.\ Akemann, Invariant subspaces of $C(G)$, \emph{Pacific J. Math.} \textbf{27} (1968), 421--424.

\bibitem{A} A.\ Atzmon, Nonfinitely generated closed ideals in group algebras, \emph{J. Funct. Anal.} \textbf{11} (1972), 231--249.

\bibitem{AH}  S.\ Argyros and R.\ Haydon, A hereditarily indecomposable $\mathscr{L}^{\infty}$--space that solves the scalar-plus-compact problem, \emph{Acta Math.}  \textbf{206} (2011), 1--54. 

\bibitem{BK} D.\ Blecher and T. Kania, Finite generation in C*-algebras and Hilbert C*-modules, \emph{Studia Math.} \textbf{224} (2014), 143--151.

\bibitem{Cas} P.\ G.\ Casazza, Approximation properties, in \emph{Handbook of the geometry of Banach spaces} pages 271--316, North Holland, Amsterdam, 2001.

\bibitem{Casto} K.\ Casto, Are convolution algebras ever ``topologically Noetherian''?, \emph{Mathoverflow}, https://mathoverflow.net/questions/185741/are-convolution-algebras-ever-topologically-noetherian.

\bibitem{C} R.\ Choukri, A concept of finiteness in topological algebras, in \emph{Topological algebras and applications} pages 131--137, Contemp. Math., \textbf{427}, Amer. Math. Soc., Providence, RI, 2007.

\bibitem{Dal}  H.\ G.\ Dales, \emph{Banach algebras and automatic continuity}, Volume \textbf{24} of \emph{London Mathematical Society Monographs, New Series}, Oxford Science Publications. The Clarendon Press, Oxford University Press, New York, 2000.

\bibitem{Da} M.\ Daws, \emph{Multipliers, self-induced and dual Banach algebras}, Dissertationes Math., Volume \textbf{470}, 2010.

\bibitem{D} P.\ Dixon, left approximate identities in algebras of compact operators on Banach spaces, \emph{Proc. Roy. Soc. Edinburgh Sect. A} \textbf{104} (1986),  169--175. 

\bibitem{DKKKL} H.\ G.\ Dales, T.\ Kania, T.\ Kochanek, P.\ Koszmider and N.\ J.\ Laustsen, Maximal left ideals in the Banach algebra of operators 
on a Banach space, \emph{Studia Math.} \textbf{218} (2013), 245--286.

\bibitem{DZ} H.\ G.\ Dales and W.\ \.Zelazko, Generators of maximal left ideals in Banach algebras, \emph{Studia Math.} \textbf{212} (2012), 173--193.

\bibitem{E} P.\ Eymard, L'alg\`ebre de Fourier d'un group localement compact, \emph{Bull. Soc. Math. France} \textbf{92} (1964), 181--236.

\bibitem{F} G.\ Folland, \emph{A course in abstract harmonic analysis}, Studies in Advanced Mathematics, CRC Press, Boca Raton, FL, 1995.

\bibitem{Gr}  M.\ Grosser, $L^{\, 1}(G)$ as an ideal in its second dual space, \emph{Proc. Amer. Math. Soc.} \textbf{73} (1979), 363--364.

\bibitem{G} N.\ Gr{\o}nb{\ae}k, Morita equivalence for Banach algebras, \emph{J. Pure Appl. Algebra} \textbf{99} (1995), 183--219.

\bibitem{GW} N.\ Gr{\o}nb{\ae}k and G.\ Willis, Approximate identities in Banach algebras of compact operators. \emph{Canad. Math. Bull.} \textbf{36} (1993),  45--53. 

\bibitem{HA1} B.\ Hayati and M.\ Amini, Connes--amenability of multiplier Banach algebras, \emph{Kyoto J. Math.} \textbf{50} (2010), 41--50.

\bibitem{HA2} B.\ Hayati and M.\ Amini, Dual multiplier Banach algebras and Connes--amenability, \emph{Publ. Math Debrecen} \textbf{86} (2015), 169--182.

\bibitem{HR1} E.\ Hewitt and K.\ A.\ Ross, \emph{Abstract harmonic analysis. Vol. I: Structure of topological groups. Integration theory, group representations}, Die Grundlehren der mathematischen Wissenschaften, Bd. 115. Academic Press, Inc., Publishers, New York; Springer-Verlag, Berlin-G\"ottingen-Heidelberg, 1963.

\bibitem{HR2} E.\ Hewitt and K.\ A.\ Ross, \emph{Abstract harmonic analysis. Vol. II: Structure and analysis for compact groups. Analysis on locally compact Abelian groups}, Die Grundlehren der mathematischen Wissenschaften, Band 152. Springer-Verlag, New York-Berlin, 1970.

\bibitem{JS}  A.\ Joyal and R.\  Street, An introduction to Tannaka duality and quantum groups. In \emph{Category theory (Como, 1990)}, 413--492, Lecture Notes in Math. \textbf{1488}, Springer, Berlin, 1991.

\bibitem{KL} E.\ Kanuith and A.\ T.--M.\ Lau, Spectral synthesis for $A(G)$ and subspaces of $VN(G)$, \emph{Proc. Amer. Math. Soc.} \textbf{129} (2001), 3253--3263. 

\bibitem{L} H.\ C.\ Lai, Multipliers of a Banach algebra in the second conjugate algebra as an idealizer, \emph{T{\^o}hoku Math. J.} \textbf{26} (1974), 431--452.
  
\bibitem{LW} N.\ J.\ Laustsen and J.\ T.\ White, Subspaces that can and cannot be the kernel of a bounded operator on a Banach space, to appear in \emph{Proceedings of the 23rd International Conference on Banach Algebras and Applications}, arXiv:1811.02393.

\bibitem{Meg} R.\ Megginson, \emph{An introduction to Banach space theory}, Graduate Texts in Mathematics \textbf{183}, Springer-Verlag, New York, 1998.

\bibitem{M} C.\ Meniri, Linearly compact rings and selfcogenerators, \emph{Rendiconti del Seminario Matematico della Universita di Padova}, \textbf{72} (1984), 99--116.

\bibitem{MPZ} P.\  Motakis, D.\ Puglisi, D.\ Zisimopoulou, A hierarchy of Banach spaces with $C(K)$ Calkin algebras, \emph{Indiana Univ. Math. J.} \textbf{65} (2016), 39--67. 

\bibitem{P} T.\ W.\ Palmer,  \emph{Banach algebras and the general theory of *-algebras. Vol. I. Algebras and Banach algebras}, Volume \textbf{49} of \emph{Encyclopedia of Mathematics and its Applications}, Cambridge University Press, Cambridge, 1994. 

\bibitem{Pr} R.\ Prosser, \emph{On the ideal structure of operator algebras}, Mem. Amer. Math. Soc. Volume \textbf{45} (1963).

\bibitem{R} V.\ Runde, Amenability for dual Banach algebras, \emph{Studia Math} \textbf{148} (2001), 47--66.
 
\bibitem{ST} A.\ M.\ Sinclair  and A.\ W.\ Tullo, Noetherian Banach algebras are finite dimensional, \emph{Math. Ann.} \textbf{211} (1974), 151--153.

\bibitem{Sz} A.\ Szankowski, $\B(H)$ does not have the approximation property, \emph{Acta Math.}, \textbf{147} (1981), 89--108.

\bibitem{S} R.\ Szwarc, Groups acting on trees and approximation properties of the Fourier algebra, \emph{J. Funct. Anal.} \textbf{95} (1991), 320--343.

\bibitem{U} A.\ \"Ulger, A characterization of the closed unital ideals of the Fourier--Stieltjes algebra $B(G)$ of a locally compact amenable group $G$, \emph{J. Funct. Anal.} \textbf{205} (2003), 90--106.

\bibitem{W1} J.\ T.\ White, Finitely-generated left ideals in Banach algebras in groups and semigroups, \emph{Studia Math.} \textbf{239} (2017), 67--99.

\bibitem{Woj} P.\ Wojtaszczyk, \emph{Banach spaces for analysts}, Cambridge Studies in Advanced Mathematics \textbf{25}, Cambridge University Press, Cambridge, 1991.

\bibitem{Young}  N.\ Young, Periodicity of functionals and representations of normed algebras on reflexive spaces, \emph{Proc. Edinburgh Math. Soc.} \textbf{20} (1976/77), 99--120.

\end{thebibliography}
\end{document}